\documentclass{amsart}

\usepackage{color}
\usepackage{graphics}
\usepackage{amsmath}
\usepackage{amsfonts}
\usepackage[colorlinks]{hyperref}
\usepackage{hyperref}
\renewcommand\eqref[1]{(\ref{#1})} %Need with hyperref

\numberwithin{equation}{section}

\newtheorem{theorem}{Theorem}[section]
\newtheorem{remark}{Remark}
\newtheorem{defin}{Definition}[section]

\usepackage{amsmath}
\usepackage{amsfonts}
\usepackage[mathcal]{eucal}	
\usepackage{amsmath}
\usepackage{amssymb}
\usepackage{amsfonts}
\usepackage{amstext}
\usepackage{amsthm}
\usepackage[mathcal]{eucal}
\usepackage{graphics}

\numberwithin{equation}{section}

\numberwithin{equation}{section}

\begin{document}

\title{ON PROJECTION FROM $A^p_\omega$ TO THE HARDY SPACES $H^p$}         % Enter your title between curly braces

\author[A. Jerbashian]{Armen Jerbashian}
\address{Armen Jerbashian:
	\endgraf
	Institute of Mathematics 
	\endgraf
	National Academy of Sciences of Armenia
	\endgraf
	24/5 Marshal Baghramian Ave., Yerevan, 0019
	\endgraf
	Republic of Armenia
		\endgraf 
	{\it E-mail address} {\rm armen\_jerbashian@yahoo.com}
}

\author[J. E. Restrepo]{Joel E. Restrepo}
\address{
	Joel E. Restrepo:
	\endgraf
	Department of Mathematics
	\endgraf 
	Cinvestav
	\endgraf
	Av. IPN 2508, Col. San Pedro Zacatenco, Mexico city, 07360.
	\endgraf
	Mexico
	\endgraf
	{\it E-mail address} {\rm joel.restrepo@cinvestav.mx}
}

\date{}          % Enter your date or \today between curly braces

\maketitle

\begin{abstract}
This paper describes the known results on the projection from the most general holomorphic spaces $A^p_\omega$, which depend on a functional parameter $\omega$ and are over the unit disc, upper half-plane and the finite complex plane, to the classical Hardy spaces $H^p.$ The paper can be considered as a survey in the mentioned topic. A new result on the projection of the half-plane $A^p_\omega$ to the half-plane Hardy space $H^p$ is obtained.
\end{abstract}


\tableofcontents

\section{Introduction}

We start by the sources of the theory of spaces $A_\omega^p$, which are the natural analogs of the classical Hardy spaces $H^p$, and the known results on the projection of the spaces $A^p_\omega$ to $H^p$, the main part of which can be found in the monograph \cite{[mon]}. Note that the space of holomorphic functions, the squares of modules of the derivatives of which are summable over the unit disc ${\mathbb D}\equiv\{z\in\mathbb{C}:|z|<1\}$, was considered, perhaps for the first time, in a paper of L. Biberbach \cite{CV-[4]} (1914), where approximations by rational functions are studied (see also pp. 150-151 in the book of J.L. Walsh \cite{CV-[5]}). In 1932, W. Wirtinger \cite{CV-[6]} studied approximations in the space $H_2^\prime$ (nowadays denoted $A^2_0$) of holomorphic in ${\mathbb D}$ functions $f$ which themselves satisfy the square summability condition
$$
A^2_0\;\;(\equiv H_2^\prime):\|f\|^2=\iint_{\mathbb D}|f(z)|^2dS<+\infty,
$$
where $dS$ is the Lebesgue area measure. In \cite{CV-[6]}, in particular, it is proved the representation formula for the functions $f\in A^2_0$ and the orthogonal projection from the same type Lebesgue space $L^2_0$ to $A^2_0$ is found (see also \cite[Theorem 20]{CV-[5]}). 

After that, in 1936, A. Ghika and S. Bergman considered the space $A^2_0$ in some multiply connected domains. S. Bergman could  establish there an analog of Wirtinger's result, but the nature of these domains permitted him to prove only the existence of the corresponding reproducing kernels (see \cite[Pages 150-151]{CV-[5]} and \cite{1950}).

Later, considering the unweighted spaces $A^p_0$ with $p\neq2$ over $\mathbb D$ V.P.  Zakaryuta and V.I. Yudovich \cite{ZY} (1964) extended Wirtinger's projection theorem to case $1<p<+\infty$. Also, they revealed the form of bounded linear functionals over $A^p_0$  $(1<p<+\infty)$ and show that the dual space of $A^p_0$ is $A^q_0$ $(1/p+1/q=1)$ in the sense of isomorphism. In W. Rudin's books \cite{Rud1} (1969) and \cite{Rudin2} (1980) the same was done in the polydisc and the multidimensional unit ball of ${\mathbb C}^n$, where the extension has some explicit forms of kernels, evident in view of the results of W. Wirtinger and V.P.  Zakaryuta -- V.I. Yudovich. In Rudin's books, the extension of Wirtinger's result (i.e. \cite[Theorem 20]{CV-[5]}) was called ``Bergman projection", and after that many researchers are attributing this term and similar ones to any result on regular functions with summable modulus over the area of a complex domain. This misunderstanding is causing a great confusion on the origins of the spaces $A^p$ in the contemporary literature.
 
Coming to weighted classes and spaces of functions in $\mathbb D$, note that in 1945 M.M. Djrbashian \cite{CV-[1]} (see also the detailed and complemented version \cite{CV-[2]} of 1948) improved R. Nevanlinna's result of 1936 (see \cite{CV-[3]}, Sec. 216) on the density of zeros and poles of functions $f$ meromorphic in $\mathbb D$, for which the Riemann--Liouville fractional integral of the growth characteristic $T(r,f)$ is bounded. The improvement was a detailed $\text{\rm factorization}$ formula for meromorphic in $\mathbb D$ functions of the considered classes. The same works \cite{CV-[1], CV-[2]} (the English translation of \cite{CV-[2]} is given in Addendum to the monograph  \cite{[mon]}) contain also the construction of the theory of similar Hardy type spaces $H^p(\alpha)$ (known as $A^p_\alpha$ nowadays) of holomorphic in $\mathbb{D}$ functions for which the Riemann-Liouville fractional primitive of the integral means $M_p(r,f)$ is bounded:
$$
A^p_\alpha\;(\equiv H^p(\alpha)):\int_0^1(1-r)^\alpha M_p(r,f)rdr\equiv\frac1{2\pi}\iint_{|\zeta|<1}\!\!(1-|\zeta|)^\alpha|f(\zeta)|^p d\sigma(\zeta)<+\infty,
$$
where $\alpha\in(-1,+\infty)$ and $p\in[1,+\infty)$ are any fixed numbers and $d\sigma$ is for Lebesgue's area measure. Besides, we recall the monograph by A.E. Djrbashian and F.A. Shamoyan \cite{CV-[15]} (1988), where the Wirtinger--Djrbashian orthogonal projection $\text{\rm theorem}$ for $A^2_\alpha$ is extended to the spaces $A^p_\alpha$ $(-1<\alpha<+\infty)$ and M.M. Djrbashian's theory of $A^p_\alpha$ spaces is complemented by numerous new results. It is to be mentioned that the results on weighted spaces of in a sense  regular, area-integrable over the unit disc functions remain of considerable interest, since they find development and application in numerous contemporary investigations, some of which are described in \cite{survey}.

The basic theory of the general, functional parameter, weighted spaces $A^p_\omega$ of holomorphic functions integrable over the areas of the unit disc, half-plane and the finite complex plane, is created in the articles \cite{CV, CMFT} and summarized in the recent monograph \cite{[mon]}. In particular, \cite{[mon]} contains some representations of functions from $A^p_\omega$ by means of some Cauchy type integrals taken over the areas of the domains, also some other integral representations over the unit circle and the real axis are given, which for $p=2$ define isometries between $A^2_\omega$ and the Hardy spaces $H^2$ over the unit disc $\mathbb D$ and upper half-plane $G^+=\{z\in\mathbb{C}:\hbox{\rm Im }z>0\}$.

The new result established in this paper is on the projection from $A_\omega^p$ over $G^+$ $(1\leqslant p<+\infty)$ to the corresponding Hardy space $H^p$, which complements the results of \cite{CV-[6],CV-[1], CV-[2],ZY,CV-[15], CV, CMFT} regarding the 
projections from $A_\omega^p$ to $H^p$ and fills a gap in the theory.

The paper is organized as follows: In Section \ref{unit-disc-s}, we recall the main results on the projection operator from $A_\omega^p(\mathbb{D})$ to $H^p(\mathbb{D})$. Here we provide some new examples to give more clarity and scope of the results. In Section \ref{finite-complex-plane}, we give the statements which are established for the projection from $A_\omega^p$ to $H^p$ in the finite complex plane. Some examples are discussed as well. In the last Section \ref{semi}, we prove the new results on the projection operator from $A_\omega^p(G^+)$ to $H^p(G^+)$. In this case, the injectivity of the operator $L_{\omega_1}$ in $G^+$ (see formula \eqref{ope-g+}) and a special bound of the kernel $C_{\omega}$ (\cite[Lemma 3.1]{CMFT-16}) play an important role. Some illustrative examples are also given.

\section{Spaces $A^p_\omega({\mathbb D})$ of functions holomorphic in the unit disc}\label{unit-disc-s}

\noindent We start by recalling the definition of the spaces $A^p_\omega({\mathbb D})$ \cite{CV} (see also \cite[Section 2]{[mon]}). First, $A_\omega^p({\mathbb D})$ is the set of all those functions $f$ holomorphic in $\mathbb D$, for which
\begin{equation}\label{CV-e1.1}
\|f\|_{p,\omega}=\biggl\{{1\over 2\pi}\iint_{|\zeta|<1} |f(\zeta)|^p |d\mu_\omega(\zeta)|\biggr\}^{1/p}<+\infty,\quad 0<p<+\infty,
\end{equation}
where $d\mu_\omega(\rho e^{i\vartheta})=-d\omega(\rho^2)d\vartheta$ and $\omega$ is a functional parameter from the class $\Omega_A({\mathbb D})$ defined as follows. 

\begin{defin}\label{defOmega_A}
$\Omega_A({\mathbb D})$ is the class of functions $\omega$ in $[0,1]$, such that:
\begin{itemize}
\item[{(i)}] \quad$0<\bigvee\limits_\delta^1\omega<+\infty$\quad for any $\delta\in[0,1)$,
\item[{(ii)}]  \quad${\Delta_n\equiv\Delta_n(\omega)=-\int_0^1 t^n d\omega(t)\neq0,\quad n=1,2,\dots,}$
\medskip\item[{(iii)}]  \quad$\displaystyle{\liminf_{n\to\infty}\root n\of{|\Delta_n|}\geqslant1}$.
\end{itemize}
$L^p_{\omega}({\mathbb D})$ is the Lebesgue space of measurable functions defined solely by \eqref{CV-e1.1}.
\end{defin}
To functions $f$ given in $\mathbb D$, we apply the M.M. Djrbashian fractional integration of the form
\begin{equation}\label{operatorD}
L_\omega f(z)=-\int_0^1f(tz)d\omega(t),\quad z\in{\mathbb D}.
\end{equation}
The operator $L_\omega$ is a modification (see in \cite{CV-[14]} or \cite[Lemma 1.1]{[mon]}) of the same operator introduced by M.M. Djrbashian \cite{CV-[11]} (see also in \cite{CV-[13]}). In some particular cases, $L_\omega$ becomes the classical integral operators of Riemann--Liouville, Hadamard \cite{Hadamard}, Erd\'elyi \cite{Erdely} -- Kober \cite{Kober} and many other operators.

By $H^p({\mathbb D})$, we denote the Hardy space of functions $f$ holomorphic in ${\mathbb D}$:
$$H^p({\mathbb D}):\quad\|f\|_{H^p({\mathbb D})}=\sup_{0<r<1}\bigg\{\int_0^{2\pi}\big|f(re^{i\vartheta})\big|^pd\vartheta\bigg\}^{1/p}<+\infty,\quad 0<p<+\infty.$$
Assuming that $\omega\in\Omega_A({\mathbb D})$, we use the Cauchy-Djrbashian  $\omega$-kernel \cite{CV-[11],CV-[13]}
$$C_\omega(z)=\sum_{n=0}^{+\infty}\frac{z^n}{\Delta_n(\omega)},\quad \Delta_n(\omega)=n\int_0^1x^{n-1}\omega(x)dx,\quad z\in{\mathbb D},$$
where
$$\Delta_n(\omega)=-\int_0^1x^{n}d\omega(x),\quad n=0,1,\ldots,\quad{\rm if}\quad\omega(1)=0.$$
Note that for any $\omega\in\Omega_A({\mathbb D})$ the holomorphicity of $C_\omega$ in $\mathbb D$ is proved in  \cite[Theorem 1.4]{CV}. Besides, we can check that 
\[
L_\omega C_\omega(z)=\frac1{1-z},\quad z\in{\mathbb D},
\]
and, if $\omega\in\Omega_A({\mathbb D})$, then $\displaystyle\lim_{n\to+\infty}\root n\of{|\Delta_n|}=1$, and hence the application of $L_\omega$ to any function $f$ holomorphic in $\mathbb D$ does not change the convergence radius of its Taylor series expansion and means just multiplication of the coefficients by $\Delta_n(\omega)$, i.e.
$${\rm if}\quad f(z)=\sum_{n=0}^{+\infty}a_kz^k\quad (z\in{\mathbb D}),\quad {\rm then}\quad
L_\omega f(z)=\sum_{n=0}^{+\infty}a_n\Delta_n(\omega)z^n\quad (z\in{\mathbb D}),$$
while application of the converse operator $L_\omega^{-1}$ means devision of the coefficients by $\Delta_n(\omega)$, and again the convergence radius does not change. More details about the existence and even on the form of the operator $L_\omega^{-1}$ can be found in \cite{CV,[mon]}.

Also, for any $p>0,$ the union $\bigcup_{\omega\in\Omega_A({\mathbb D})}A^p_\omega({\mathbb D})$ coincides with the set of all functions holomorphic in $\mathbb D$.

If $f\in A_\omega^p({\mathbb D})$ $(1\leqslant p<+\infty)$ for some $\omega\in\Omega_A({\mathbb D})$, then the following representations are true at any point $z\in\mathbb D$ \cite[Theorem 1.4]{CV}:
\begin{align}
f(z)&={1\over 2\pi}\iint_{|\zeta|<1} f(\zeta) C_\omega(z\overline\zeta) d\mu_\omega(\zeta)\label{D3}\\
&=-\overline{f(0)}+{1\over\pi}\iint_{|\zeta|<1} \big\{{\hbox{\rm Re }}f(\zeta)\big\}\, C_\omega(z\overline\zeta) d\mu_\omega(\zeta).\nonumber
\end{align}
In \cite[Theorem 3.1]{CV}, it is also proved that for $p=2$ formula 
$$f(z)={1\over 2\pi}\iint_{|\zeta|<1}F(\zeta) C_\omega(z\overline\zeta) d\mu_\omega(\zeta),\quad F\in L^2_\omega({\mathbb D}),$$
defines the orthogonal projection from $L^2_\omega({\mathbb D})$ to $A^2_\omega({\mathbb D})$.

So, the associated Cauchy type kernel of the space $A^p_\omega({\mathbb D})$ is $C_\omega$, and when $\omega(t)=(1-t)^\alpha$ $(\alpha>0)$, then $C_\omega(z)=(1-z)^{-(1+\alpha)}$. Besides, in \cite[Theorem 1.5]{CV}, it is established that a similar representation formula for functions of $A^2_\omega({\mathbb D})$ gives an isometry with the Hardy space $H^2({\mathbb D})$. 

\begin{theorem}\label{CV-t1.2} 
{\bf (\cite[Theorem 1.5]{CV})}
Let $\widetilde\omega\in\Omega_A({\mathbb D})$ be continuously differentiable, nonincreasing in $(0,1]$, $\widetilde\omega(1)=0$ and $\widetilde\omega(0)=1$, and let $\omega$ be the Volterra square of $\widetilde\omega$, i.e.
$$\omega(x)=-\int_x^1\widetilde\omega\Bigl({x\over\sigma}\Bigr) d\widetilde\omega(\sigma),\quad 0<x<1.
$$
Then $\omega\in\Omega_A({\mathbb D})$, and $A_{\omega}^2({\mathbb D})$ coincides with the set of functions representable as
\begin{equation}\label{CV-e1.5}
f(z)={1\over 2\pi}\int_0^{2\pi}C_{\widetilde\omega}(ze^{-i\vartheta})\varphi(e^{i\vartheta})d\vartheta,\quad z\in{\mathbb D},\;\;\varphi(e^{i\vartheta})\in L^2[0,2\pi].
\end{equation}
For any $f\in A_{\omega}^2({\mathbb D})$, there exists a unique function $\varphi_0$ of the Hardy space $H^2({\mathbb D})$, such that \eqref{CV-e1.5} is true with $\varphi_0(e^{i\vartheta})$. This function can be deduced by the formula
$$\varphi_0(z)=L_{\widetilde\omega} f(z)=-\int_0^1 f(tz) d\widetilde\omega(t),\quad z\in{\mathbb D}.$$
Besides, $\|\varphi_0\|_{H^2({\mathbb D})}=\|f\|_{2,\omega}$ and $\varphi-\varphi_0\perp H^2({\mathbb D})$ for any $\varphi(e^{i\vartheta})\in L^2[0,2\pi]$ for which \eqref{CV-e1.5} is true. The operator $L_{\widetilde\omega}$ is an isometry $A^2_{\omega}({\mathbb D})\to H^2({\mathbb D})$, and the integral of \eqref{CV-e1.5} defines $L_{\widetilde\omega}^{-1}$ on $H^2({\mathbb D})$.
\end{theorem}
The next theorem presents a result on the projection from $A_\omega^p(\mathbb{D})$ to $H^p(\mathbb{D})$ for any $1\leqslant p<+\infty$. 
\begin{theorem}\label{t1} {\bf (\cite[Theorem 5.5]{CV})}
Let $1\leqslant p<+\infty$ and let $\omega\in\Omega_A({\mathbb D})$ be a monotone function. Then:
\begin{itemize}
\item[{$1^\circ$.}]\; Any function $f(z)\in A_\omega^p({\mathbb D})$ is representable in the form
\begin{equation}\label{d3}
f(z)={1\over2\pi}\int_0^{2\pi}C_{\omega_1}(ze^{-i\vartheta})\varphi(e^{i\vartheta})d\vartheta,\quad z\in{\mathbb D},
\end{equation}
where $\omega_1(x)=\omega(x^2)$ $(0\leqslant x\leqslant 1)$ and $\varphi=L_{\omega_1}f\in H^p({\mathbb D})$ with the integral operator of the form \eqref{operatorD}. 
\medskip\item[{$2^\circ$.}] \;
The linear operator $L_{\omega_1}:A_\omega^p(\mathbb{D})\to H^p(\mathbb{D})$ is bounded such that $\|L_{\omega_1}\|\leqslant 1,$ and \eqref{d3} represents $L^{-1}_{\omega_1}$ in the set $L_{\omega_1}A_\omega^p({\mathbb D})\,(\subset H^p({\mathbb D}))$.
\end{itemize}
\end{theorem}

\subsection{Examples}
\medskip $1^\circ.$ If $1<p<+\infty$ and $\omega_1(t)=(1-t)^\alpha$ with $\alpha\geqslant1$, then 
$$C_{\omega_1}(z)\equiv(1-z)^{-(1+\alpha)}\quad{\rm and}\quad L_{\omega_1}f(z)\equiv\alpha\int_0^1(1-t)^{\alpha-1}f(tz)dt,$$
but 
$$\omega(t)=\alpha^{p}\int_t^1(1-x)^{(\alpha-1)p}dx=\frac{\alpha^{p}}{(\alpha-1)p+1}(1-t)^{(\alpha-1)p+1}.$$
Hence, formula \eqref{d3} is true for any function $f$ of the space $A^p_{\omega}({\mathbb D})$ having the associated Cauchy type kernel 
$$C_{\omega}(z)=\frac1{(1-z)^{(\alpha-1)p+2}},$$
while the order $1+\alpha$ of the Cauchy kernel in the representation \eqref{d3} can be much smaller than $(\alpha-1)p+2$, i.e. the representation \eqref{d3} can be much better than \eqref{D3}.

\medskip $2^\circ.$ Let $1<p<+\infty$, and let $\omega_1$ be a continuously differentiable, strictly monotone function in $[0,1]$. Then, the function 
$$\omega(t)=\int_t^1\;\big|\omega_1^\prime(\lambda)\big|^{p}d\lambda$$
is of the class $\Omega_A({\mathbb D})$, and any function $f\in A_\omega^p({\mathbb D})$ is representable in the form
\[
f(z)={1\over2\pi}\int_0^{2\pi}C_{\omega_1}(ze^{-i\vartheta})\varphi(e^{i\vartheta})d\vartheta,\quad z\in{\mathbb D},
\]
where $\varphi=L_{\omega_1}f\in H^p({\mathbb D})$ with the integral operator of the form \eqref{operatorD}. Besides, the linear operator $L_{\omega_1}:A_\omega^p(\mathbb{D})\to H^p(\mathbb{D})$ is bounded such that $\|L_{\omega_1}\|\leqslant 1,$ and \eqref{d3} represents $L^{-1}_{\omega_1}$ in the set $L_{\omega_1}A_\omega^p({\mathbb D})\,(\subset H^p({\mathbb D}))$.

\section{Spaces $A^p_\omega({\mathbb C})$ of entire functions}\label{finite-complex-plane}

In this section, we focus on the theory of the spaces of entire functions $A^p_\omega({\mathbb C})$, which is given in \cite{CV} (see also \cite[Section 3]{[mon]}). 

The spaces $A^p_\omega({\mathbb C})$ $(0<p<+\infty)$ are defined as the sets of entire functions $F$ for which
\begin{equation}\label{CV-e5.1}
\|F\|_{p,\omega}=\left\{{1\over2\pi}\iint_{\mathbb C}|F(\zeta)|^p|d\mu_\omega(\zeta)|\right\}^{1/p}<+\infty,
\end{equation}
where $d\mu_\omega(\rho e^{i\vartheta})=-d\omega(\rho^2)d\vartheta$ and $\omega\in\Omega_A({\mathbb C})$, i.e. $\omega$ is a strictly decreasing in the half-axis $[0,+\infty)$, such that $\omega(0)=1$ and
$$\Delta_n^\infty(\omega)=-\int_0^{+\infty}t^nd\omega(t)<+\infty\quad{\hbox{\rm for any}}\quad n=0,1,2\ldots$$
By $L^p_\omega({\mathbb C})$ we denote the corresponding Lebesgue space.  $A^p_\omega({\mathbb C})$ $(p\geqslant1)$ is a Banach space with the norm \eqref{CV-e5.1}, while $A^2_\omega({\mathbb C})$ is a Hilbert space. Furthermore, for any $p>0,$ the union $\bigcup_{\omega\in\Omega_A({\mathbb C})}A^p_\omega({\mathbb C})$ coincides with the set of all entire functions \cite[Lemma 5.1]{CV}. Besides, $A^{p_1}_\omega({\mathbb C})\subset A^{p_2}_\omega({\mathbb C})$ for $0<p_1<p_2<+\infty$.

\medskip\noindent If $\omega\in\Omega_A({\mathbb C})$ then 
$$\lim_{n\to\infty}\root n\of{\Delta_n(\omega)}=+\infty \;\;\text{for}\;\;\Delta_n(\omega)=-\int_0^{+\infty}t^n|d\omega(t)|,$$
and the Cauchy type kernel of M.M. Djrbashian \cite{1970}
$$C_\omega^{\infty}(z)=\sum_{n=0}^{+\infty}\frac{z^n}{\Delta_n^\infty(\omega)},\quad z\in {\mathbb C},$$
is an entire function. Further, considering the M.M. Djrbashian \cite{1970} fractional integration operator
$$L_\omega^\infty f(z)=-\int_0^{+\infty}f(tz)d\omega(t),\quad z\in{\mathbb C},$$
one can see that:
$$L_\omega C_\omega(z)=\frac1{1-z},\quad z\in{\mathbb D},$$
and, the application of $L_\omega^\infty$ to any entire function $f$ means just multiplication of its Taylor series coefficients by $\Delta_n(\omega)$, i.e.
$${\rm if}\quad f(z)=\sum_{n=0}^{+\infty} a_kz^k\;\;\; (z\in{\mathbb C}),\quad {\rm then}\quad
L_\omega f(z)=\sum_{n=0}^{+\infty} a_n\Delta_n(\omega)z^n\;\;\; (|z|<R\leqslant+\infty),$$
while as application of the converse operator $L_\omega^{-1}$ we mean devision of the coefficients by $\Delta_n^\infty(\omega)$, of course, for the set of entire functions for which the requirement $\liminf_{n\to\infty}\root n\of{|a_n|\Delta_n(\omega)}=1/R$ is fulfilled. So, $L_\omega^\infty$ is a one-to-one mapping of the mentioned set of entire functions to functions holomorphic in the disc $|z|<R$.  

If $f\in A_\omega^p({\mathbb C})$ $(2\leqslant p<+\infty)$ for some $\omega\in\Omega_A({\mathbb C})$, then the following representations of functions from $A^p_\omega({\mathbb C})$ are proved \cite[Theorem 5.2]{CV}:
\begin{align}
F(z)&={1\over2\pi}\iint_{\mathbb C}F(\zeta)C_\omega^\infty(z\overline\zeta)d\mu_\omega(\zeta)\label{REP}\\
&=-\overline{F(0)}+{1\over\pi}\iint_{\mathbb C}\big\{{\hbox{\rm Re }}F(\zeta)\big\}C_\omega^\infty(z\overline\zeta)d\mu_\omega(\zeta),\quad z\in{\mathbb C}.\nonumber
\end{align}
\begin{remark}
Note that proving a representation like \eqref{REP} for functions $f\in A_\omega^p({\mathbb C})$ with $1\leqslant p<2$ is still an open problem, while the results  obtained on the projection are about projection operators from $A_\omega^p({\mathbb C})$ to Hardy spaces $H^p({\mathbb D})$ over the unit disc.
\end{remark}
The orthogonal projection of $L_\omega^2({\mathbb C})$ onto $A_\omega^2({\mathbb C})$ is of the form \cite[Theorem 5.3]{CV}
\[
P_\omega f(z)={1\over 2\pi}\iint_{\mathbb C}f(\zeta) C_\omega^\infty(z\overline\zeta)
d\mu_\omega(\zeta),\quad z\in{\mathbb C},\,\,\, f\in L_\omega^2({\mathbb C}).
\]
In addition, the below theorem with a similar representation formula for functions of $A^2_\omega({\mathbb C})$ is  proved in \cite[Theorem 5.4]{CV}, which gives an isometry of $A^2_\omega({\mathbb C})$ with the Hardy space $H^2({\mathbb D})$.

\begin{theorem}\label{CV-t5.3}
Let $\omega\in\Omega_A({\mathbb C})$ be continuously differentiable in $[0,+\infty)$ and such that $\omega(+\infty)=0$, $\omega^\prime<0$ and is bounded on $[0,+\infty)$ and $$-\infty<\int_0^{+\infty}t^{-1}d\omega(t)<0.$$
Then, the function
$$\widetilde\omega(x)=-\int_0^{+\infty}\omega\Bigl({x\over t}\Bigr)d\omega(t),\quad 0<x<+\infty,$$
belongs to $\Omega_A({\mathbb C})$, and $A_{\tilde\omega}^2({\mathbb C})$ coincides with the set of all functions representable in the form
\begin{equation}\label{CV-e5.7}
F(z)={1\over 2\pi}\int_0^{2\pi}\varphi(e^{i\vartheta}) C_\omega^\infty(ze^{-i\vartheta}) d\vartheta,\quad z\in{\mathbb C},\;\;\;\varphi(e^{i\vartheta})\in L^2[0,2\pi].
\end{equation}
For any $F\in A_{\tilde\omega}^2({\mathbb C})$, there exists a unique function $\varphi_0$ of the ordinary Hardy space $H^2({\mathbb D})$, such that \eqref{CV-e5.7} is true with $\varphi_0(e^{i\vartheta})$. This function can be found by the formula
$$\varphi_0(z)=L_\omega^\infty F(z)\equiv-\int_0^{+\infty}F(tz)d\omega(t),\quad z\in{\mathbb D}.$$
Besides, $\|\varphi_0\|_{H^2({\mathbb D})}=\|F\|_{2,\tilde\omega}$ and $\varphi-\varphi_0\perp H^2$ for any 
$\varphi(e^{i\vartheta})\in L^2[0,2\pi]$ for which \eqref{CV-e5.7} is true.  The operator $L_\omega^\infty$ is an isometry $A^2_{\tilde\omega}({\mathbb C})\to H^2({\mathbb D})$, and the integral \eqref{CV-e5.7} defines $(L_\omega^\infty)^{-1}$ on $H^2({\mathbb D})$.
\end{theorem}
For the projection from $A_\omega^p(\mathbb{C})$ to $H^p(\mathbb{D})$ with any $1\leqslant p<+\infty$, the following statement is true.
\begin{theorem}\label{tE} {\bf (\cite[Theorem 5.5]{CV})}
Let $1\leqslant p<+\infty$, let $\omega\in\Omega_A({\mathbb C})$ and let $\bigvee_0^{+\infty}\omega=1$. Then any function $F\in A^p_\omega({\mathbb C})$ is representable in the form
\begin{equation}\label{CV-e5.8}
F(z)={1\over2\pi}\int_0^{2\pi}C_{\omega_1}^\infty(ze^{-i\vartheta})\varphi(e^{i\vartheta})d\vartheta,\quad z\in{\mathbb C},
\end{equation}
where $\omega_1(x)=\omega(x^2)$ and $\varphi=L_{\omega_1}^\infty F\in H^p({\mathbb D})$. In $A^p_\omega({\mathbb C})$ we have $\|L_{\omega_1}^\infty\|\leqslant1$, and \eqref{CV-e5.8} represents $\big(L^\infty_{\omega_1}\big)^{-1}$ in the set $L_{\omega_1}^\infty  A^p_\omega({\mathbb C})$.
\end{theorem}

\subsection{Examples} 
$1^\circ.$ Let $1<p<+\infty$, $1/p+1/q=1$, and let $\omega_1$ be a continuously differentiable, strictly decreasing function in $[0,+\infty)$, such that for some $0<\varepsilon\leqslant 1$
$$M_{p,\varepsilon}\equiv\bigg\{\int_0^{+\infty}\big[-\omega_1^\prime(t)\big]^{q(1-\varepsilon)}dt\bigg\}^{1/q}<+\infty\;\;\; and\;\;\;\int_0^{+\infty}\big[-\omega_1^\prime(\lambda)\big]^{p\varepsilon}d\lambda=1.$$
Then, the function $\omega(t)=\int_t^{+\infty}\big[-\omega_1^\prime(\lambda)\big]^{p\varepsilon}d\lambda$ is of the class $\Omega_A({\mathbb C})$. So, by Theorem \ref{tE}, any function $f\in A_\omega^p({\mathbb C})$ is representable in the form
\begin{equation}\label{aaa}
f(z)={1\over2\pi}\int_0^{2\pi}C_{\omega_1}^{\infty}(ze^{-i\vartheta})\varphi(e^{i\vartheta})d\vartheta,\quad z\in{\mathbb C},
\end{equation}
where $\varphi=L_{\omega_1}f\in H^p({\mathbb D}),$ $\|L_{\omega_1}\|\leqslant M_{p,\varepsilon}$ in $A_\omega^p({\mathbb C})$, and \eqref{aaa} represents $L^{-1}_{\omega_1}$ in the set $L_{\omega_1}A_\omega^p({\mathbb C})\,(\subset H^p({\mathbb D}))$.

\medskip $2^\circ.$ Note that the statements of Theorem \ref{CV-t5.3} remain valid in the case when
\[
\omega^\prime(x)=-C_0\,e^{-\gamma x^\rho}x^{\mu\rho-1},\quad 0<x<+\infty,\quad\text{with}\quad
 C_0=\Bigl(\int_0^{+\infty}e^{-\gamma x^\rho}x^{\mu\rho-1}dx\Bigr)^{-1}
\]
for some $\gamma,\rho,\mu>0$, though the requirements on $\widetilde\omega^\prime$ providing the inclusion $\omega\in\Omega_A({\mathbb C})$ in Theorem \ref{CV-t5.3} are not satisfied. In this case, the kernel $C_\omega^\infty$ in the representation \eqref{CV-e5.7} becomes the Mittag-Leffler function  \cite[Theorem 5.6]{CV}.

\section{Spaces $A^p_\omega(G^+)$ of functions holomorphic in the upper half-plane}\label{semi}

In this section, we provide the new results of the paper on the projection operator from $A_\omega^p(G^+)$ to $H^p(G^+)$ for any $1\leqslant p<+\infty$. We first recall some concepts and properties from the theory of the spaces $A^p_\omega$ in the upper half-plane $G^+$ (see \cite{CMFT-16,CMFT}, and also \cite[Sections 5 and 6]{[mon]}). 

\begin{defin}\label{d.1}%{\bf (Definition 5.1 in \cite{[mon]})}
 $A^p_{\omega,\gamma}(G^+)$ $(0<p<+\infty,$ $-\infty<\gamma\leqslant2)$ is the set of those functions $f$ holomorphic in the upper half-plane $G^+$, which for sufficiently small $\rho>0$ satisfy the Nevanlinna condition
$$
\liminf_{R\to+\infty}\frac{1}{R}\int_\beta^{\pi-\beta}\log^+|f(Re^{i\vartheta})|\left(\sin\frac{\pi(\vartheta-\beta)}{\pi-2\beta}\right)^{1-\pi/\varkappa}d\vartheta=0
$$
where $\beta=\arcsin\frac{\rho}{R}=\frac{\pi}{2}-\varkappa$ and, simultaneously,
\begin{equation}\label{Khavinson-e2}
\|f\|^p_{p,\omega,\gamma}\equiv\iint_{G^+}|f(z)|^p\frac{d\mu_\omega(z)}{(1+|z|)^\gamma}<+\infty,
\end{equation}
where $d\mu_\omega(x+iy)=dxd\omega(2y)$ and $\omega$ is of a class $\Omega_\alpha(G^+)$ $(-1\leqslant\alpha<+\infty)$, i.e. $\omega$ is given in $[0,+\infty)$ and such that:
\begin{itemize}
\item[(i)] $\omega\nearrow$ in $(0,+\infty)$,  $\omega(0)=\omega(+0)$ and $\omega(\delta_k)\downarrow$ (is strictly decreasing) on a 
sequence $\delta_k\downarrow0$;
\smallskip\item[(ii)]\; $\omega(t)\asymp t^{1+\alpha}$ for $\Delta_0\leqslant x<+\infty$ and some $\Delta_0\geqslant 0$ $(f\asymp g$ means that $m_1f\leqslant g\leqslant m_2f$ for some constants $m_{1,2}>0)$.
\end{itemize}
\smallskip $L^p_{\omega,\gamma}(G^+)$ is the Lebesgue space defined solely by \eqref{Khavinson-e2}.
\end{defin}
\noindent Everywhere below, we consider the spaces $L^p_{\omega,\gamma}(G^+)$ and $A^p_{\omega,\gamma}(G^+)$ only with $\gamma=0,$ and therefore we denote them briefly $L^p_{\omega}(G^+)$ and $A^p_{\omega}(G^+)$ respectively. Further, to functions $f$ given in $G^+$, we apply the operator
\begin{equation}\label{ope-g+}
L_\omega f(z)=\int_0^{+\infty}f(z+it)d\omega(t),\quad z\in G^+.
\end{equation}
Assuming that $\omega\in\Omega_\alpha(G^+)$ $(-1\leqslant\alpha<+\infty)$, we use the Cauchy type $\omega$-kernel
$$C_\omega(z)=\int_0^{+\infty}e^{itz}\frac{dt}{I_\omega(t)},\quad I_\omega(t)=t\int_0^{+\infty}e^{-tx}\omega(x)dx,\quad z\in G^+,$$
where $I_\omega(t)$ can also be rewritten as
$$I_\omega(t)=\int_0^{+\infty}e^{-tx}d\omega(x),\quad{\rm if}\quad\omega(0)=0.$$
It is known that for any $\omega\in\Omega_\alpha(G^+)$ $(-1\leqslant\alpha<+\infty),$ the function $C_\omega$ is holomorphic in $G^+$ \cite{CMFT-16}. Furthermore, one can prove that
$$L_\omega C_\omega(z)=-\frac1{iz},\quad z\in G^+.$$
By $H^p(G^+)$, we denote the Hardy space of functions $f$ holomorphic in $G^+$ and such that the following condition holds
$$\|f\|_{H^p(G^+)}=\sup_{0<y<+\infty}\bigg\{\int_{-\infty}^{+\infty}|f(x+iy)|^pdx\bigg\}^{1/p}<+\infty,\quad 0<p<+\infty.$$ 
If $f\in A^p_{\omega}(G^+)$ for some $1\leqslant p<+\infty$, $\omega\in\Omega_\alpha(G^+)$ $(-1\leqslant\alpha<+\infty)$, then by \cite[Theorem 3.1]{CMFT-16}, the following integral representations are true:
\begin{align*}
f(z)&=\frac{1}{2\pi}\iint_{G^+}f(w)C_\omega(z-\overline w)d\mu_\omega(w)\\
&=\frac{1}{\pi}\iint_{G^+}\{{\hbox{\rm Re }}f(w)\}C_\omega(z-\overline w)d\mu_\omega(w),\quad z\in G^+,
\end{align*}
where the integrals are absolutely and uniformly convergent inside $G^+$. Also, by \cite[Theorem 4.1]{CMFT-16}, for $p=2,$ we have that the formula 
$$f(z)=\frac{1}{2\pi}\iint_{G^+}F(w)C_\omega(z-\overline w)d\mu_\omega(w),\quad F\in L^2_{\omega}(G^+),$$
defines the orthogonal projection operator from $L^2_{\omega}(G^+)$ to $A^2_{\omega}(G^+)$. Besides,  \cite[Theorem 6]{CMFT} establishes a similar representation formula for functions of $A^2_{\omega}(G^+)$, which gives an isometry with the Hardy space $H^2(G^+)$. We recall the latter assertion in the next result.

\begin{theorem}\label{CMFT-t6}
Let $\omega\in\Omega_\alpha(G^+)$ with $-1\leqslant\alpha<+\infty$ and $\omega(0)=0$. Then, the Volterra square 
$$\widetilde\omega(x)=\int_0^x\omega(x-t)d\omega(t),\quad 0<x<+\infty,\quad \widetilde\omega(0)=0,$$
of $\omega$ belongs to $\Omega_{1+2\alpha}$, and $A^2_{\widetilde\omega}(G^+)$ coincides with the set of functions representable in the form
\begin{equation}\label{CMFT-e5.13}
f(z)=\frac1{2\pi}\int_{-\infty}^{+\infty}\varphi(t)C_\omega(z-t)dt,\quad z\in G^+,\;\; \varphi\in L^2(-\infty,+\infty).
\end{equation}
For any $f\in A^2_{\widehat\omega}(G^+)$, $L_\omega f\equiv\varphi_0$ is the unique function of the Hardy $H^2(G^+)$, such that \eqref{CMFT-e5.13} is true with $\varphi=\varphi_0$. Besides, $\|\varphi_0\|_{H^2(G^+)}=\|f\|_{A^2_{\widetilde\omega}(G^+)}$ and $\varphi-\varphi_0\bot H^2(G^+)$ for any $\varphi\in L^2(-\infty,+\infty)$ with which \eqref{CMFT-e5.13} is true. The operator 
$$L_\omega f(z)\equiv\int_0^{+\infty}f(z+i\sigma)d\omega(\sigma),\quad z\in G^+,$$
is an isometry $A^2_{\widetilde\omega}(G^+)\longrightarrow H^2(G^+)$, and the integral in \eqref{CMFT-e5.13} defines $(L_\omega)^{-1}$ in $H^2(G^+)$.
\end{theorem} 
Now, in the next theorem, we give answer to the remain open question about the projection operator from the general spaces $A_\omega^p$ to the Hardy space $H^p$ in the upper half-plane $G^+$ for any $1\leqslant p<+\infty$. In this case, we then complement the previous result and partially resolve this question in this space without claiming optimality of the bounds and the functional parameter $\omega$. 
 
\begin{theorem}\label{tG+} 
\begin{itemize}
\item[{$1^\circ.$}] Let $1<p<+\infty$ and let $\omega_1$ be a continuously differentiable, strictly increasing function in some interval $[0,\Delta]\subset[0,+\infty)$ and constant in $[\Delta,+\infty)$, i.e. $\omega_1\in\Omega_{-1}(G^+)$ with $\omega_1(t)=\omega_1(\Delta)$ for $\Delta\leqslant t<+\infty$. Further, let $\omega^\prime(t)$ be a non-decreasing function in $[0,\Delta]$. Then, any function $f\in A_\omega^p(G^+)$ is representable in the form
\begin{equation}\label{fG}
f(z)=\frac1{2\pi}\int_{-\infty}^{+\infty}C_{\omega_1}(z-t)\varphi(t)dt,\quad z\in G^+.
\end{equation}
where $\varphi=L_{\omega_1}f\in H^p(G^+)$. Besides, $\|L_{\omega_1}\|\leqslant \Delta_0^{p-1}$ in $A_\omega^p(G^+)$, and \eqref{fG} represents $L^{-1}_{\omega_1}$ in the set $L_{\omega_1}A_\omega^p(G^+)\,(\subset H^p(G^+))$.
\item[{$2^\circ.$}] Let $\omega\in\Omega_\alpha(G^+)$ $(-1\leqslant\alpha<+\infty)$ and $\omega_1(t)=\omega(2t)$ $(0<t<+\infty)$. Then, any function $f\in A^1_\omega(G^+)$ is representable in the form \eqref{fG}. Besides, $\|L_{\omega_1}\|\leqslant1$ in $A_\omega^1(G^+)$, and formula \eqref{fG} represents $L^{-1}_{\omega_1}$ in the set $L_{\omega_1}A_\omega^1(G^+)$ $(\subset H^1(G^+))$.
\end{itemize}
\end{theorem}
\begin{proof} 
$1^\circ.$ If $f\in A_{\omega}^p(G^+)$, then obviously the function $L_{\omega_1}f$ is holomorphic in $G^+$. Besides, by Jensen's inequality and the hypothesis of $\omega^\prime(t)$ being a non-decreasing function in $[0,\Delta],$ it yields 
\begin{align*}
\frac1{2\pi}\int_{-\infty}^{+\infty}&\big|L_{\omega_1}f(x+iy)\big|^pdx\leqslant
\frac{\Delta_0^p}{2\pi}\int_{-\infty}^{+\infty}\bigg[\frac{1}{\Delta_0}\int_0^{\Delta_0}\big|f(x+iy+it)\big|\omega_1^\prime(t)dt\bigg]^pdx\\
&\leqslant\frac{\Delta_0^{p-1}}{2\pi}\int_{-\infty}^{+\infty}dx\int_0^{\Delta_0}\big|f(x+iy+it)\big|^p[\omega_1^\prime(t)]^{p}dt\\
&=\frac{\Delta_0^{p-1}}{2\pi}\int_{-\infty}^{+\infty}dx\int_y^{\Delta_0+y}\big|f(x+it)\big|^p[\omega_1^\prime(t-y)]^{p}dt\\
&\leqslant\frac{\Delta_0^{p-1}}{2\pi}\int_{-\infty}^{+\infty}dx\int_0^{+\infty}\big|f(x+it)\big|^p[\omega_1^\prime(t)]^{p}dt\\
&=\frac{\Delta_0^{p-1}}{2\pi}\int_{-\infty}^{+\infty}dx\int_0^{+\infty}\big|f(x+it)\big|^pd\omega(2t)=
\Delta_0^{p-1}\|f\|_{L^p_{\omega}(G^+)}^p
\end{align*}
for any $0<y<+\infty$. Thus,
$$\|L_{\omega_1}f\|_{H^p(G^+)}^p=\sup_{0<y<+\infty}\biggl\{\frac1{2\pi}\int_{-\infty}^{+\infty}\big|L_{\omega_1}f(x+iy)\big|^pdx\Biggr\}\leqslant \Delta_0^{p-1}\|f\|_{L^p_{\omega}(G^+)}^p<+\infty.$$
Hence, the following Cauchy formula is true:
$$L_{\omega_1}f(z)=\frac1{2\pi i}\int_{-\infty}^{+\infty}\frac{L_{\omega_1}f(t)}{t-z}dt,\quad z\in G^+,$$
where $L_{\omega_1}f(t)$ are the boundary values of the function $L_{\omega_1}f\in H^p(G^+)$ on the real axis. Using the estimate
$$|C_{\omega_1}(z)|\leqslant M_\varepsilon|z|^{-1},\quad \hbox{\rm Im }z>\varepsilon,$$
which is true for any $\varepsilon>0$ and a suitable constant $M_\varepsilon>0$ (see \cite[Lemma 2.2]{CMFT-16}), and H\"older's inequality, one can prove that the function
$$g(z)\equiv\frac1{2\pi}\int_{-\infty}^{+\infty}C_{\omega_1}(z-t)L_{\omega_1}f(t)dt,\quad z\in G^+,$$
is holomorphic in $G^+$. Also, we have that for any $z\in G^+$
$$L_{\omega_1}g(z)=\frac1{2\pi}\int_{-\infty}^{+\infty}L_{\omega_1}C_{\omega_1}(z-t)L_{\omega_1}f(t)dt=
\frac1{2\pi i}\int_{-\infty}^{+\infty}\frac{L_{\omega_1}f(t)}{t-z}dt=L_{\omega_1}f(z).$$
We already know that the operator $L_{\omega_1}$ is a one-to-one mapping in the class of functions holomorphic in $G^+$, and hence $f\equiv g$ in $G^+$, i.e. the representation \eqref{fG} is true. %###

\noindent$2^\circ$. If $f\in A_\omega^1(G^+)$, then obviously the function $L_{\omega_1}f$ is holomorphic in $G^+$, and for any $0<y<+\infty$
\begin{align*}
\frac1{2\pi}\int_{-\infty}^{+\infty}\big|L_{\omega_1}f(x+iy)\big|dx&\leqslant\frac1{2\pi}\int_{-\infty}^{+\infty}dx\int_0^{+\infty}\big|f(x+it)\big|d\omega_1(t)\\
&=\frac1{2\pi}\int_{-\infty}^{+\infty}dx\int_0^{+\infty}\big|f(x+it)\big|d\omega(2t)=\|f\|_{L^1_\omega(G^+)}.
\end{align*}
Thus, $L_{\omega_1}f$ is of the Hardy space $H^1(G^+)$, and hence
$$L_{\omega_1}f(z)=\frac1{2\pi}\int_{-\infty}^{+\infty}\frac{L_{\omega_1}f(t)}{t-z}dt,\quad z\in G^+.$$
Using the estimate 
$$|C_\omega(z)|\leqslant\frac{M}{|z|^{2+\beta}},\quad z\in G^+_\rho,$$
which is true for any non-integer $\beta\in([\alpha]-1,\alpha)$, any $\rho>0$ and a constant $M\equiv M_{\rho,\beta}>0$ (see \cite[Lemma 3.1]{CMFT-16}), so the function
$$g(z)\equiv\frac1{2\pi}\int_{-\infty}^{+\infty}C_{\omega_1}(z-t)L_{\omega_1}f(t)dt,\quad z\in G^+,$$
is holomorphic in $G^+$, and it follows for any $z\in G^+$ that
$$L_{\omega_1}g(z)=\frac1{2\pi}\int_{-\infty}^{+\infty}L_{\omega_1}C_{\omega_1}(z-t)L_{\omega_1}f(t)dt=
\frac1{2\pi i}\int_{-\infty}^{+\infty}\frac{L_{\omega_1}f(t)}{t-z}dt=L_{\omega_1}f(z).$$
Again, since the operator $L_{\omega_1}$ is a one-to-one mapping in the class of functions holomorphic in $G^+$, it follows that  $f\equiv g$ in $G^+$, i.e.
$$f(z)=\frac1{2\pi}\int_{-\infty}^{+\infty}C_{\omega_1}(z-t)L_{\omega_1}f(t)dt,\quad z\in G^+,$$
which finishes the proof.
\end{proof}

\subsection{Examples}

Some classical examples of the functional parameters $\omega\in\Omega_\alpha(G^+)$ are functions behaving, for example, as $\omega(t)=e^{t}-1$, $\omega(t)=\log(1+t)$ and $\omega(t)=t^{1+\alpha}$ $(\alpha>0)$ in a finite interval $[0,\Delta)$.

\section*{Conflict of interest}

The authors of this work declare that they have no conflicts of interest.

\end{document}